\documentclass[11pt, oneside]{amsart}
\usepackage{mathtools}
\usepackage{wrapfig}
\usepackage{tikz}
\usepackage{comment}
\usepackage{mathrsfs}
\usepackage{tabularx, hyperref}
\usepackage{amssymb} \usepackage{amsfonts} \usepackage{amsmath}
\usepackage{amsthm} \usepackage{epsfig, subfig}
\usepackage{ amscd, amsxtra, latexsym}
\usepackage{epsfig,  graphicx, psfrag}
\usepackage[all]{xy}
\usepackage{caption}
\usepackage{enumerate}
\usepackage{color}

\newtheorem{lemma}{Lemma}[section]
\newtheorem{thm}[lemma]{Theorem}
\newtheorem{prop}[lemma]{Proposition}
\newtheorem{cor}[lemma]{Corollary}

\newtheorem{prop_intro}{Proposition}

\newtheorem{quest_intro}[prop_intro]{Question}
\newtheorem{thm_intro}[prop_intro]{Theorem}

\theoremstyle{definition}
\newtheorem{rmk_intro}[prop_intro]{Remark}

\newtheorem*{prop*}{Proposition}

\newtheoremstyle{citing}
  {3pt}
  {3pt}
  {\itshape}
  {}
  {\bfseries}
  {}
  {.5em}
  {\thmnote{#3}}
\theoremstyle{citing}

\newcommand{\has}{{\scalebox{1.3}\#}}
\newcommand{\id}{\mathrm{id}}
\newcommand{\rel}{{\mathrm{rel}.\,}}

\newcommand{\R}{\ensuremath {\mathbb{R}}}

\newcommand{\csum}[2][n]{{#2_1 \,\# \,\ldots\,  \# \,#2_{#1} }}

\newcommand{\aut}{{\mathrm{Aut}}}

\renewcommand{\phi}{\varphi}

\newcommand\restr[2]{{
  \left.\kern-\nulldelimiterspace 
  #1 
  \vphantom{\big|} 
  \right|_{#2} 
  }}
\newcommand\rrestr[2]{{
  \left.\kern-\nulldelimiterspace 
  \left.\kern-\nulldelimiterspace 
  #1 
  \vphantom{\big|} 
  \right|\hspace{-2.4pt} 
  \right|_{#2} 
  }}

\begin{document}

\title[]{Homotopy Equivalences of 3-Manifolds}

\author[Federica Bertolotti]{Federica Bertolotti}
\address{Scuola Normale Superiore, Pisa, Italy}
\email{federica.bertolotti@sns.it}


\begin{abstract}
    Let $M$ be an oriented closed $3$-manifold. We prove that there exists a constant $A_M$, depending only on the manifold $M$, such that for every self-homotopy equivalence $f$ of $M$ there is an integer $k$ such that $1 \leq k \leq A_M$ and $f^k$ is homotopic to a homeomorphism.
\end{abstract}

\maketitle

\section{Introduction}
In the last century a great effort has been devoted to understanding the relations between homotopy equivalences and homeomorphisms of closed manifolds and many results have been achieved;
however, the actual situation seems far from clear.
In dimension $2$ it is well known that every orientation preserving homotopy equivalence is homotopic to a homeomorphism;
on the other hand, in greater dimension the same statement can no longer be true and in order to prove something similar some hypotheses have to be added.
In this direction there exists a famous, almost as long-lasting, unproven conjecture, known as \emph{Borel Conjecture}, stating that every homotopy equivalence between oriented closed aspherical manifolds is homotopic to a homeomorphism.
Even if this conjecture is still open, much progress has been done, at least in dimension $n \neq 4$ (\cite{bartels_borel_2012},~\cite{lafont_special_2022}).

Given a topological space $X$, we denote by $\mathcal E (X)$ the group of homotopy classes of self-homotopy equivalences of $X$ and by $\mathcal H(X)$ the subgroup of $\mathcal E(X)$ containing the classes representable by a homeomorphism;
according to the Borel conjecture, if $M$ is a closed aspherical manifold, then $\mathcal E(M)= \mathcal H(M)$.

The space $\mathcal E(X)$ is an interesting object on its own and it has been studied for many kinds of spaces (see \cite{rutter_spaces_1997},~\cite{costoya_primer_2020} for a general introduction), such as $3$-manifolds (\cite{canary_homotopy_2004},~\cite{kreck_topological_2005}), odd dimensional spheres (\cite{smallen_group_1974},~\cite{kishimoto_monoids_2021}), products of spheres (\cite{baues_group_1996}, \cite{kreck_topological_2005} again), rational elliptic spaces (\cite{costoya_every_2014},~\cite{benkhalifa_group_2020}).

We are mainly interested in $3$-dimensional manifolds.
In this specific case the Borel conjecture was proved in 2005 by Kreck and L\"uck and, even more, in the same paper it is shown that if the fundamental group of an oriented closed connected $3$-manifold $M$ is torsion free, then again $\mathcal E(M)= \mathcal H(M)$ (\cite[Theorem 0.7]{kreck_topological_2005}).
It should be noted that every oriented closed prime $3$-manifold not homeomorphic to a quotient of the $3$-sphere $S^3$ has torsion free fundamental group.

On the other hand, there are examples of prime $3$-manifolds, all homeomorphic to (connected sums of) quotients of $S^3$, admitting homotopy equivalences that are not homotopic to any homeomorphism (\cite{mccullough_mappings_1986}).
However, if $M$ is a quotient of $S^3$, then the group $\mathcal E(M)$ is finite (see~\cite{smallen_group_1974} for quotients of the $3$-sphere or~\cite{kishimoto_monoids_2021} for a more general statement about quotients of spheres of odd dimensions);
thus, every homotopy equivalence $f \colon M \to M$ admits a power $f^k$ homotopic to a homeomorphism or, even more, the subgroup $\mathcal H(M)$ has finite index inside $\mathcal E(M)$.

It is then natural to ask whether a similar statement holds for every oriented closed $3$-manifold $M$.

\begin{quest_intro}
    Let $M$ be an oriented closed $3$-manifold $M$. Does the subgroup $\mathcal H(M)$ have finite index inside $\mathcal E(M)$?
\end{quest_intro}

In this direction we show that every homotopy equivalence $f \colon M \to M$ has a power $f^k$ homotopic to a homeomorphism and, moreover, the minimal integer $k>0$ satisfying this property is bounded by a constant $A_M$ depending only on the manifold $M$.

\begin{thm_intro}\label{main}
    Let $M$ be an oriented closed $3$-manifold $M$. Then, there exists a constant $A_M$ such that for every homotopy equivalence $f \colon M \to M$ there exists an integer $k$ satisfying $1\leq k \leq A_M$ and such that $f^k$ is homotopic to a homeomorphism.
\end{thm_intro}
Let $M$ be an oriented closed connected $3$-manifold. We describe here a constant $A_M$
satisfying the previous theorem.

If $M = S^3$, then it suffices to take $A_M = 1$.

If $M$ is not homeomorphic to the sphere $S^3$, then let $M = \csum M$ be a prime decomposition;
up to rearranging the indices, we can assume $M_1,\,\ldots \,,\,M_r$ have finite fundamental groups $\pi_1(M_1),\,\ldots \,,\,\pi_1(M_r)$, with cardinality $c_1,\,\ldots \,,\,c_r$ respectively, and $M_{r +1},\,\ldots \,,\,M_n$ have infinite fundamental groups;
then we can set
\[A_M = 2 \cdot r !\cdot (c_1 \cdot c_2\,\cdots \,c_r )!\,.\]

\begin{rmk_intro}[Simplifications on the hypotheses of Theorem \ref{main}]\label{hypsim}

    As the square of a homotopy equivalence is always orientation preserving, it is enough to prove Theorem \ref{main} with the further hypothesis on $f$ of being orientation preserving (we just need to remember to add an additional coefficient $2$ to the expression of $A_M$ at the end of the proof).

    Moreover, if $M$ has $q>1$ connected components, then there exists an integer $i$ such that $1 \leq i \leq q !$ and $f^i$ sends each connected component to itself.
    Thus, it suffices to prove the statement for oriented closed \emph{connected} manifolds.
\end{rmk_intro}

\subsection{Plan of the paper}
In Section \ref{preliminaries:sec} we give a very quick overview about decompositions of $3$-manifolds, distinguishing between two kinds of prime summands:
the \emph{elliptic} ones, consisting of quotients of the $3$-sphere $S^3$, and the \emph{nonelliptic} ones, that are all the other prime summands.
We also recall the definition of splitting homotopy equivalences and we state a result by Hendriks and Laudenbach asserting that orientation preserving homotopy equivalences always split along decomposing systems of spheres (\cite{hendriks_scindement_1974}).

In Section \ref{primesummand} we see how the results in \cite{kreck_topological_2005} and \cite{smallen_group_1974} imply that Theorem \ref{main} holds for every prime $3$-manifold.

In Section \ref{proof} we prove the main theorem:
first, given an orientation preserving self-homotopy equivalence $f \colon M \to M $ of an oriented closed connected $3$-manifold $M$, we show that there exists a sequence of powers $f^{\beta_k}$ of $f$ preserving some decomposing system of spheres $\Sigma$ and sending each piece of the decomposition induced by $\Sigma$ into a homeomorphic piece;
out of this sequence we extract a power of $f$ that is homotopic to a homeomorphism on each summand and, thus, it is itself homotopic to a homeomorphism.
\subsection{Acknowledgements}
I am grateful to my Ph.D. supervisor, Roberto Frigerio, for pointing out the problem and some references;  I would like to thank Matteo Migliorini and Francesco Milizia for important suggestions about Lemma \ref{GR} and for useful comments on previous versions of this paper. I would also like to thank Giuseppe Bargagnati, Pietro Capovilla and Domenico Marasco for useful discussions.

\section{Decomposition along spheres}\label{preliminaries:sec}

\subsection{Decompositions of $3$-manifolds}\label{dec}
We recall here basic facts about decompositions of $3$-manifolds in prime summands and we fix some notation.
More details can be found e.g.~in~\cite{martelli_introduction_2016}.

Let $M$ be an oriented closed connected $3$-manifold.
A \emph{separating sphere} in $M$ is an embedded $2$-sphere $S \subset M$ such that $M \setminus S$ consists of two connected components and a \emph{decomposing system of spheres} is a disjoint union of separating spheres in $M$.

If $\Sigma = S_1 \,\sqcup \, \ldots \, \sqcup S_{n -1} \subset M$ is a decomposing system of spheres, then $M \setminus \Sigma$ consists of exactly $n$ connected components:
indeed, every time a sphere is removed, one component is split into two.
Let us denote by $M_1',\,\ldots \,,\, M_n'$ the closure of these connected components; in this way
every $M_i'$ is a compact manifold whose boundary consists of a disjoint union of spheres, all contained in $\Sigma$;
we denote by $M_i$ the oriented closed connected manifold obtained from $M_i'$ by gluing a $3$-ball $B$ to every boundary component $S \in \{S_1,\,\ldots\,,\, S_{n -1} \}$ of $M_i'$ according to an orientation reversing identification $\partial B \cong S$.
According to this construction, we have
\[M = \csum M \,. \]

In this context we say that $M$ is \emph{decomposed along} $\Sigma$ or that the \emph{decomposition is induced} by $\Sigma$.

\subsubsection{Prime decomposition}
A closed oriented $3$-manifold $P$ not homeomorphic to $S^3$ is called \emph{prime} if every decomposition of $P$ is trivial (i.e.~if $P = M_1 \, \#\,M_2$, then $M_i = S^3$ for $i = 1$ or $i = 2$).
Given an oriented closed $3$-manifold $M$, we call \emph{prime decomposition} of $M$ a decomposition $M =M_1 \,\# \,\ldots \,\#\,M_n$ in which each summand $M_i$ is prime.

Recall that every oriented closed connected $3$-manifold $M$ not homeomorphic to $S^3$ admits a unique prime decomposition, where by unique we mean that whenever
\[
    M \,=\,
    M_{1}^{1}\,\#\,M_{2}^{1}\,\#\,\ldots \, \#\, M_{n}^{1} \,=\,
    M_{1}^{2} \,\#\,M_{2}^{2}\,\#\,\ldots \, \#\, M_{m}^{2}
\]
are two prime decompositions, then $m = n$ and there exists a permutation $\sigma \in \mathcal S_n$ such that $M_i^1$ is homeomorphic to $M^2_{\sigma(i)}$ for every $i \in \{1,\,\ldots \,,\,n \}\,$.

We distinguish between prime $3$-manifolds with finite fundamental group and those with infinite fundamental group:
the former are all obtained as quotients of $S^3$ and are usually referred as \emph{elliptic manifolds};
a prime $3$-manifold with infinite fundamental group is called \emph{nonelliptic prime manifold}: these spaces have noncompact universal cover (that is homeomorphic either to  $S^2 \times \R$ or to $\R^3$) and torsion free fundamental group.

\subsection{Connected sum of homotopy equivalences}\label{csum}
Together with manifolds, we also need to split homotopy equivalences along spheres;
however, to understand what a splitting is, we should first understand what it means to  sum maps.
We provide in the following the definition of connected sum of orientation preserving homotopy equivalences.

For $i= 1,\,2$, let us consider an orientation preserving homotopy equivalence $f_i \colon N_i \to M_i$ between two oriented closed connected manifolds $M_i,\,N_i$.
Up to homotopy, we can suppose there exist two closed balls $B_i \subset M_i,\,C_i\subset N_i$ such that $f_i$ maps homeomorphically (and preserving the orientation) the ball $C_i$ to the ball $B_i$.
Let $h_i$ be a homotopy inverse of $f_i$; since the degree of $f_i$ is $1$, we can choose $h_i$ in such a way that $h_i^{-1} (C_i)= B_i$ and it sends $B_i$ homeomorphically to $C_i$.
As in the previous subsection, let $M_i'$ be the closure of $M_i \setminus B_i$ and $N_i'$ the closure of $N_i \setminus C_i$ and fix an orientation reversing identification $\partial C_1 \cong \partial C_2$.

Since $\restr{f_1} {C_1}$ and $\restr{f_2} {C_2}$ are orientation preserving homeomorphism, there is a well defined orientation reversing identification $\partial B_1 \cong \partial B_2$ such that whenever two points $x_1 \in C_1,\,x_2 \in C_2$ are identified, then $f_1(x_1)\in B_1$ and $f_2(x_2) \in B_2$ are identified as well.
Let us define
\begin{align*}
    N &=\, N_1 \, \# \, N_2 \,=\, N_1'\, \bigcup_{\partial C_1 \cong \partial C_2} \,N_2'\,,      \\
    M &= M_1 \,\# \,M_2 = M_1'\, \bigcup_{\partial B_1 \cong \partial B_2} \,M_2'\,,
\end{align*}
which are oriented closed connected manifolds.

In this setting, the connected sum
\[ f_1 \# f_2 \colon N_1 \, \# \, N_2 \to M_1 \# M_2 \, \]
is defined by gluing $\restr{f_1} {N'_1}$ and $\restr{f_2} {N'_2}$ along  $\partial C_1 \cong \partial C_2$.
This map is a well defined orientation preserving homotopy equivalence:
indeed, from cellular approximation it follows that $h_1 \,\#\, h_2$ is a homotopy inverse of $f$.

In the following, whenever we write a connected sum $f_1 \,\# \, f_2$ between orientation preserving homotopy equivalences, we always have a (possibly implicit) choice of balls $B_1,\,B_2,\,C_1,\,C_2$ as above and an identification $\partial C_1 \cong \partial C_2$. On the other hand, the homotopy class of $f_1 \, \# \,f_2$ does not depend on these choices (this is a consequence of cellular approximation and the fact that all balls inside a connected $3$-manifold are isotopic).

If more than two orientation preserving homotopy equivalences are given, then the connected sum is defined by summing iteratively starting from the last two terms;
namely, if $f_i \colon N_i \to M_i$ are orientation preserving homotopy equivalences for $i \in \{1,\,\ldots \,,\,n \}$, then
\[
    f_1 \,\#\,\ldots \,\#\,f_n \,\coloneqq\,
    f_1 \,\#\left(f_2 \, \#\left(\,\ldots \,\#\left(f_{n -1}\#f_n \right)\cdots \right) \right)\,.
\]
As before, this is a well defined orientation preserving homotopy equivalence whose homotopy class does not depend on the choices done.

For $i \in \{1,\,\ldots \,,\,n \}$, let $g_i \colon N_i \to M_i$ be another orientation preserving homotopy equivalence and suppose that the connected sum $\csum g$ is well defined.
If $f_i$ and $g_i$ are homotopic for every $i \in \{1,\,\ldots \,,\,n \}$, then
the two maps $\csum g$ and $\csum f$ are homotopic as well.

Analogously, if the map $f_i$ is homotopic to a homeomorphism for every $i \in \{0,\,\ldots \,,\,n\}$, then the map $\csum{f}$ is homotopic to a homeomorphism.
\subsection{Splitting homotopy equivalences}
We are now ready to define the splitting of a homotopy equivalence and to state a result by Hendriks and Laudenbach assuring that every orientation preserving self-homotopy equivalence of an oriented closed connected $3$-manifold splits (up to homotopy) along any separating sphere.

Let $M,\,N$ be two oriented closed connected $3$-manifolds and let $\Sigma = S_1 \,\sqcup \, \ldots \, \sqcup S_{n -1} \subset  M$ be a decomposing system of spheres in $M$ inducing the decomposition $M = \csum M$. An orientation preserving homotopy equivalence $f \colon N \to M $ \emph{splits} along $\Sigma$ if
\begin{itemize}
    \item for every $i \in \{1,\,\ldots \,,\, n -1 \}$, the map $f$ is transverse to $S_i$ and $S_i'= f^{-1} (S_i)$ is an embedded $2$-sphere in $N$,
    \item $\Sigma'= S_1'\sqcup \,\ldots \,\sqcup S'_{n -1}$ is a decomposing system of spheres inducing a decomposition $N =\csum{N}$,
    \item up to rearranging the indices, for every $i \in \{1,\,\ldots \,,\, n  \}$ there exists an orientation preserving homotopy equivalence $f_i \colon N_i \to M_i$  such that $f =\csum{f}\,$, where the connected sum is taken along the decomposing system of spheres $\Sigma' $ (meaning that the connected sum is obtained by gluing the maps $f_1,\,\ldots \,,\, f_{n -1}$ along the spheres $S_1',\,\ldots \,,\,S_{n -1}'$ according the definition given in Subsection \ref{csum}).
\end{itemize}
    
In this case we say that $f$ sends the decomposition along $\Sigma'$ to the decomposition along $\Sigma$ and we denote by
\[\rrestr{f}{N_i}\colon N_i \to M_i \]
the map $f_i$. Let us highlight that $\rrestr{f} {N_i}$ is not a restriction of $f$ and $N_i$ is not a subset of $N$.

Whenever we have a splitting homotopy equivalence $f \colon N \to M$ sending a decomposition $N =\csum{N}$ to a decomposition $M =\csum{M}$, we always assume that the map $\rrestr{f} {N_i}$ goes from $N_i$ to $M_i$ (with the same index).

In addition, if $N = M$ and for every $i \in \{1 ,\,\ldots \,,\, n \}$ the summands $N_i$ and $M_i$ are homeomorphic, then we say that $f$ \emph{preserves the types of homeomorphism} of the decomposition induced by $\Sigma$;
notice that preserving the types of homeomorphism of a decomposition does not imply acting as a homeomorphism on every summand of the decomposition.

An orientation preserving homotopy equivalence $g \colon N \to M$ \emph{homotopically splits} along $\Sigma$ if it is homotopic to a homotopy equivalence $f \colon N \to M$ that splits along $\Sigma$.

It was proved by Hendriks and Laudenbach in ~\cite{hendriks_scindement_1974} that every orientation preserving homotopy equivalence between $3$-manifolds homotopically splits along an embedded nontrivial sphere.
\begin{thm}[\cite{hendriks_scindement_1974}]\label{HL}
    Let $M$ and $N$ be two oriented closed connected $3$-manifolds and let $S \subset M$ be an embedded separating sphere not homotopic to a point. Then every orientation preserving homotopy equivalence $f \colon N \to M$ homotopically splits along $S$.
\end{thm}

In their paper, Hendriks and Laudenbach considered also the case in which $S$ is not separating inside $M$.

We now extend this result to decompositions with more than two summands.
\begin{cor}\label{splitting:lemma}
    Let $M$ and $N$ be two oriented closed connected $3$-manifolds and $\Sigma$ a decomposing system of spheres of $M$ inducing a decomposition $M =\csum M$ with no summands homeomorphic to $S^3$.
    Then every orientation preserving homotopy equivalence $f \colon N \to M$ homotopically splits along $\Sigma$.
\end{cor}
\begin{proof}
    The proof is by induction on the number of spheres of $\Sigma$.
    
    If $\Sigma$ is empty, then there is nothing to prove; so, suppose
    \[\Sigma = S_1 \sqcup \ldots \sqcup S_n \]
    consists of $n>0$ spheres as in the statement and that the corollary is true for every decomposing system of spheres satisfying the hypotheses and with  less than $n$ spheres.
    Let
    \[M = \csum[n +1]{M}\]
    be the decomposition induced by $\Sigma$;
    without loss of generality we can assume  $S_n$ is a boundary component of both $M'_{n+1}$ and $M'_n$ (we are adopting here the same notation introduced in Subsection \ref{dec}) and set $Y = M_{n }\,\# \,M_{n +1}$.
    
    The decomposing system of spheres
    \[\Sigma_1= S_1 \,\sqcup \,\ldots \, \sqcup \, S_{n -1}\,,\]
    with $n -1$ components, induces the decomposition
    \[M = \csum[n -1]{M}\,\#\, Y \,,\]
    where, by construction, none of the summands is homeomorphic to $S^3$.
    By the inductive hypothesis the map $f$ homotopically splits along $\Sigma_1$: there exist a decomposition
    \begin{equation}\label{decN}
        N = N_1 \, \#\,\ldots \, \#\,N_{n -1}\, \# \, X \,
    \end{equation}
    and orientation preserving homotopy equivalences $f_i \colon N_i \to M_i$ (for $i \in \{1,\,\ldots \,,\, n -1 \}$) and $h \colon X \to Y$ such that the map $f$ is homotopic to
    \[ g_1 = \csum[n -1]{f} \,\#\, h \]
    and $\Sigma_1'=g_1^{-1}(\Sigma_1)= S_1'\,\sqcup \,\ldots \,\sqcup \, S_{n -1}'$ is a decomposing system of spheres inducing the decomposition (\ref{decN}).
    
    Since $Y = M_{n }\,\#\,M_{n +1}$ is a decomposition along the sphere $S_n$, Theorem \ref{HL} yields a splitting of the map $h \colon X \to Y$ along $S_n$:
    we get two $3$-manifolds $N_{n }$ and $N_{n +1} $ and two orientation preserving homotopy equivalences $f_{n }\colon N_{n} \to M_{n} $ and $f_{n +1} \colon N_{n +1} \to M_{n +1}$  such that $X = N_{n } \# N_{n +1}$, the map $h$ is homotopic to $h'=f_{n} \,\#\,f_{n +1}\,$, and $(h')^{-1} (S_n)= S_n'$ is an embedded separating sphere.
    
    Now we need to modify slightly the map $h$ and the sphere $S_n'$ in order to guarantee the well-definition of the connected sum $\csum[n +1] {f}$ along the spheres $S_1 ,\,\ldots \,,\,S_n$.

    Since $\Sigma_1'$ is a decomposing system of spheres, every sphere in $S_i'\subset X \cap \Sigma_1'$ bounds a ball in $X$; thus we can homotope $h$ in such a way that the sphere $S_n'= h^{-1} (S_n)$ does not intersect $\Sigma_1'$.


    Moreover, if $S_i \in \{S_1,\,\ldots \,,\,S_{n -1}\}$ is contained in $M_n$ or in $M_{n +1}$, (let us say in $M_n$), then we can suppose $f_n$ splits along $S_i$ (\cite[Proposition 1.1]{hendriks_scindement_1974}) so that $f_n^{-1} (S_i)$ is a sphere inside $N_n$.
    Thus, up to homotopy, we can suppose $\restr{h} {S_i'} =\restr{f_n} {S_i'}$.
    By repeating this process for all the spheres in $\Sigma_1\cap Y$, we obtain that the connected sum $\csum[n+1] f$ is well defined.

    Recollecting all the pieces together, one gets that:
    \begin{itemize}
            \item $f$ is homotopic to $\csum[n -1]{f}\,\# \,h$ that, in turn, is homotopic to
                \[g =\csum[n -1]{f}\,\#\,h'=\csum[n+1] f \,, \]
                where each $f_i \colon N_i \to M_i$ is an orientation preserving homotopy equivalence;
            \item the union
                \[\Sigma'= g^{-1} (\Sigma)= \Sigma'_1\,\sqcup \,S_n'= S'_1 \,\sqcup \,\ldots \, \sqcup \,S'_n \]
                is a decomposing system of spheres;
            \item the system $\Sigma'$ induces the decomposition
                \[N =\csum[n +1]{N}\,.\]
    \end{itemize}
    Therefore the map $f$ homotopically splits along $\Sigma$.
\end{proof}

\section{Orientation preserving self-homotopy equivalences of prime $3$-manifolds}\label{primesummand}
Although Theorem \ref{main} is not obvious in general, for prime manifolds it is an easy consequence of already existing results.
We recall here why this is true, distinguishing between the elliptic and nonelliptic case.

\subsection{Nonelliptic manifolds.}
An oriented closed connected manifold $M$ is called \emph{strong Borel manifold} if every orientation preserving homotopy equivalence $f \colon N \to M$ from an oriented closed connected manifold $N$ to $M$ is homotopic to a homeomorphism $g \colon N \to M$.

The following result is due to Kreck and L\"uck:
    \begin{thm}[\cite{kreck_topological_2005}]\label{sbm}
        If $M$ is an oriented closed connected $3$-manifold with torsion free fundamental group, then $M$ is a strong Borel manifold.
    \end{thm}
Every nonelliptic prime manifold has torsion free fundamental group and thus it is a strong Borel manifold.
Moreover, if $M = \csum{M}$ is a prime decomposition of an oriented closed connected $3$-manifold $M$, then the fundamental group of $M$ is given by
\[\pi_1(M)=\pi_1(M_1)\,*\,\ldots \, *\,\pi_1(M_n)\,.\]
Thus, $\pi_1(M)$ is torsion free if and only if for every $i \in \{1,\,\ldots \, ,\,n \}$ the group $\pi_1(M_i)$ is torsion free or, equivalently, the prime manifold $M_i$ is nonelliptic.

\subsection{Elliptic manifolds.}
As mentioned in the introduction, an elliptic $3$-manifold  is not necessarily a strong Borel manifold, but it is still true that orientation preserving self-homotopy equivalences defined on  it have a power homotopic to a homeomorphism.
Before of proving this result, let us introduce some notation.

Let $M$ be an oriented closed connected manifold (not necessarily of dimension $3$), $p\in M$ a base point and $f,g\colon (M,p) \to  (M,p)$ two self-homotopy equivalences of $M$ fixing the point $p$;
if there exists a homotopy between $f$ and $g$ fixing the point $p$, then we write $f \sim_{\rel p}g$ and we say that $f,g$ are homotopic relatively to $p$.
We denote by $\mathcal E(M,p)$ the group given by
\[
    \mathcal  E(M,p):=
    \{f\colon (M,p) \to (M,p) \text{ homotopy equivalence}\}/\sim_{\rel p}\,.
\]

Since every class $\phi\in\mathcal E(M,p)$ induces an automorphism
\[\phi_\#\colon \pi_1(M,p)\to \pi_1(M,p)\]
of the fundamental group $\pi_1(M,p)$ of $M$, we have a well defined homomorphism
\[\begin{matrix}
  \Theta_{M,\,p}  \colon &\mathcal E(M,p)&\to &\aut(\pi_1(M,p)) \,.\\
                  &\phi &\mapsto &\phi_\#
\end{matrix}\]
In general this map is neither injective nor surjective (\cite{rutter_spaces_1997}).

Smallen proved that whenever $M$ is an elliptic $3$-manifold, the map $\Theta_{M,\,p}$ is injective.
\begin{thm}[\cite{smallen_group_1974}]\label{sma}
    Let $G$ be a finite group that acts on $S^3$ without fixed points and let $p \in S^3/G$. Then $\Theta_{S^3/G,\,p}$ is injective.
\end{thm}
Let $M = S^3/G$.
Since the fundamental group $\pi_1(M,p) = G$ is a finite group, every element in $\aut(G)$ has order less then $c!\,$, where $c =|G|$ is the cardinality of $G$.
In particular, for every self-homotopy equivalence $f \colon (M,p) \to (M,p)$, there exists an integer $k$ satisfying $1 \leq k \leq c ! $ and $(f_\#)^k = \id_G$. Thanks to the result by Smallen, the map $f^k$ is homotopic to the identity (that is a homeomorphism).

In view of what just said, the next result is a direct consequence of the fact that every self-homotopy equivalence of a connected manifold is homotopic to a self-homotopy equivalence fixing a point.

\begin{prop}
    Let $M = S^3/G$ be an elliptic manifold and $f \colon M \to M$ a homotopy equivalence. Then there exists an integer $k$ such that $1 \leq k \leq |G|!$ and $f^k$ is homotopic to a homeomorphism.
\end{prop}
It should be noted that in most of the cases  $|G|!$ is not the optimal bound. For example every automorphism of $G$ fixes the identity element $e_G \in G$ and acts as a permutation on the remaining $|G| -1$ elements.
Thus, a better (but still very loose) bound on the order of elements in $\aut(G)$ is $(|G| -1)!\,$.
However, we believe this number is still far from being optimal: for example, it does not take into account that many homotopy equivalences are homotopic to a homeomorphism.

\section{Proof of Theorem \ref{main}}\label{proof}
So far, we proved that for every prime $3$-manifold $M$ there exists a constant $A_M$ (depending only on the manifold $M$) with the property that for every orientation preserving homotopy equivalence $f \colon M \to M$ there exists a positive integer $k \in \{1,\,\ldots \,,\, A_M \}$ such that $f^k$ is homotopic to a homeomorphism.
The constant $A_M$ is $1$ when $M$ is nonelliptic and it may be chosen to be $c!$ when $M$ is elliptic and the fundamental group $\pi_1(M)$ is a finite group of cardinality $c$.

Moreover, we saw that for every oriented closed connected $3$-manifold $M$, every decomposing system of spheres $\Sigma \subset M$ and every orientation preserving homotopy equivalence $f \colon M \to M$, the map $f$ homotopically splits along $\Sigma$.

Now we need to put these results together.
One of the main problems is that the prime decomposition is unique only up to homeomorphism:
given two prime decompositions $M = \csum{M^1}$ along $\Sigma_1$ and $M =\csum{M^2}$ along $\Sigma_2$, there exists a permutation $\sigma \in \mathcal S_n$ so that $M^1_i$ is homeomorphic to $M^2_{\sigma(i)}$ (for every $1 \leq i \leq n$);
however, there is no way to identify the spheres in $\Sigma_1$ with the spheres in $\Sigma_2$ and there is not a canonical homeomorphism between the summands of the first decomposition and the ones of the second decomposition.
Moreover, even if $f \colon M \to M$ is an orientation preserving homotopy equivalence sending the decomposition induced by $\Sigma_2$ to the decomposition induced by $\Sigma_1$, there is no guarantees that $f$ preserves the types of homeomorphism of the decomposition along the system $\Sigma_1$.

With the following lemma we take care of this problem and find powers $f^{\beta_k}$ of $f$ homotopic to some orientation preserving homotopy equivalences $g_k$ that preserve the types of homeomorphism of some decompositions.
\begin{lemma}\label{lemmalungo}
    Let $M$ be an oriented closed connected $3$-manifold, $m$ a positive integer and $f \colon M \to M$ an orientation preserving homotopy equivalence.
    Then there exist $m +1$ decomposing systems of spheres $\Sigma_0,\,\ldots ,\,\Sigma_m$, $m$ positive integers $\beta_1,\,\ldots \,,\, \beta_m $ and $m$ orientation preserving self-homotopy equivalences $g_1,\,\ldots \,,\, g_m \colon M \to M$ of $M$ such that for every $k \in\{1,\,\ldots \,,\,m  \}$ the map $g_k$
    \begin{enumerate}
        \item\label{1}
            splits along $\Sigma_{k -1}$ and $g_k^{-1} ({\Sigma}_{k -1})= \Sigma_{k}$,
        \item\label{2}
            preserves the types of homeomorphism of $\Sigma_{k -1}$,
        \item\label{3}
            is homotopic to $f^{\beta_{k}}$.
    \end{enumerate}
    Moreover, the integers $\beta_k$ satisfy $\sum_{k = 1}^m \beta_k \leq m \cdot r !\,$, where $r$ is the number of elliptic manifolds appearing in a prime decomposition of $M$.

\end{lemma}
Before proving this lemma, we should remark that if a homotopy equivalence $g$ sends the decomposing system of spheres $\Sigma_2$ to the system $\Sigma_1$, then we do not know whether the map $g$ (or some homotopic map) splits along $\Sigma_2$ or whether there exists a system $\Sigma_0$ on which $g$ (or some homotopic map) splits and such that $g^{-1} (\Sigma_0)=\Sigma_1$.
Thus, if the map $g$ preserves the types of homeomorphism of $\Sigma_1$ and it is homotopic to $f^\beta$ for some integer $\beta>0$, then we cannot directly obtain the desired result by taking the maps $g^2,\,g^3,\,\ldots \,$, as in general they are not ``compatible'' with the system $\Sigma_1$ nor with the system $\Sigma_2$.
Therefore, we need a more delicate (and technical) argument.

The statement is still true when $m =\infty$.

\begin{proof}
    When $M = S^3$ the statement is clear, so let us suppose $M$ is not homeomorphic to $S^3$ and fix a prime decomposition
    \[ M = M_1^0\, \#\, M_2^0 \,\#\, \ldots \,\#\, M_n^0\]
    induced by a decomposing system of spheres $\Sigma'_0 \subset M$.

    \emph{Step 1.}
    For every integer $k \geq 1$ we will construct by recursion a decomposing system of spheres $\Sigma'_k \subset M$, inducing a prime decomposition $M =\csum{M^k}$, and a homotopy equivalence $f_k \colon M \to M$, homotopic to $f$ and sending the decomposition along $\Sigma'_{k}$ to the decomposition along $\Sigma'_{k -1}$.

    Suppose we constructed them for every integer $k \in \{1,\,\ldots \,,\,t -1 \}\,$.
    By Corollary $\ref{splitting:lemma}$ there exists a homotopy equivalence $f_{t}$ homotopic to $f$ and splitting along $\Sigma'_{t -1}$; set $\Sigma'_{t}= f_{t}^{-1} (\Sigma'_{t -1})$, that is a decomposing system of spheres inducing a decomposition
    \[ M = M_1^{t} \,\#\, M_2^{t}\, \#\, \ldots \,\# \,M_n^{t}\,. \]
    As for every $s \in \{ 1,\,\ldots \,,\,n \}$ the map $\rrestr{f_{t}}{ M_s^{t}}\colon M_s^{t}\to M_s^{t -1}$ is a homotopy equivalence and $M_s^{t -1}$ is not homeomorphic to $S^3$ (by definition of prime decomposition), then $M_s^{t}$ is not homeomorphic to $S^3$ either;
    hence $\Sigma_{t }$ induces a prime decomposition.

    Let us set
    \[ f_{s,k}=\rrestr{f_k}{ M_{s}^{k}}\colon M_{s}^{k}\to M_s^{k -1}\,. \]
    As for $k \in \{1,\,\ldots \,,\,  t \}$ the map $f_k = \has_{s = 1}^{n} \left(f_{s,k}\right)$ is homotopic to $f$, the map $f^t$ is homotopic to
    \[
        f_1 \,\circ \, \ldots \, \circ \, f_t =
        \mathop{\has} \limits_{s = 1}^{n}\left(f_{s,1}\, \circ \,\ldots \, \circ \,  f_{s,t}\right)\,.
    \]

    \emph{Step 2.}
    We will find a sequence of integers $(\alpha_k )_{0 \leq k \leq m }$ for which $M_s^{\alpha_k}$ and $M_s^{\alpha_j} $ are homeomorphic for every $j,k \in\{0,\,\ldots \,,\,m \}$ and every $s \in \{1 ,\,\ldots \,,\, n \}$.

    Set $M_s:= M_{s}^{0}\,$;
    up to rearranging the indices, we can suppose $M_s$ is elliptic for $ s \in \{1,\,\ldots \,,\,r \}$ and nonelliptic for $ s \in \{r +1 ,\,\ldots \,,\, n \}\,$.

    Since nonelliptic prime manifolds are strong Borel manifolds, for every $k>0$ and $s \in \{r +1,\,\ldots \,,\, n \}$ the orientation preserving homotopy equivalence $f_{s,1} \,\circ \, \ldots \,\circ  \,f_{s,k}$ is homotopic to a homeomorphism and the summand $M_s^{k}$ is homeomorphic to $M_s$.

    On the other hand, because of the uniqueness of prime decompositions, for every integer $k \geq 0$ there exists a permutation $\sigma_k \in \mathcal S_{r}$ such that for every $s \in \{1,\,\ldots \,,\, r \}$ the prime summand $M_s^{k}$ is homeomorphic to $M_{\sigma_k(s)}$.
     Since the elements in $\mathcal S_r $ are exactly $r !\,$, the pigeonhole principle provides a permutation in $\mathcal S_r$ appearing at least $m +1$ times in $(\sigma_k )_{0 \leq k \leq m\cdot r !}\,$;
     let us denote by $\tau$ this permutation and let
    \[0 \leq \alpha_0<\alpha_1<\ldots<\alpha_m \leq m \cdot r ! \]
    be such that $\sigma_{\alpha_k}= \tau$ for every $k \in \{0 ,\,\ldots \,,\, m \}\,$.
    
    By construction, for every $s \in \{1,\,\ldots \,,\,r \}$ and $j,k \in \{0,\,\ldots \,,\, m \}$, the prime summand $M_s^{\alpha_j}$ is homeomorphic to $M_{\sigma_{\alpha_j} (s)}= M_{\tau (s)} =M_{\sigma_{\alpha_k} (s)}$, which, in turn, is homeomorphic to $M_s^{\alpha_k}.$
    
    \emph{Conclusion.}
    Let us take $\Sigma_k =\Sigma'_{\alpha_k}$ (for $0 \leq k \leq  m$), $\beta_k =\alpha_{k} - \alpha_{k - 1}$ (for $1 \leq k \leq m$) and
    \[
        g_k=
        f_{\alpha_{k -1}+1}\circ f_{\alpha_{k -1}+2}\circ\ldots \circ f_{\alpha_{k}}=
        \mathop{\has}\limits_{s = 1}^{n}\left(
            f_{s,\alpha_{k -1}+1}\circ f_{s,\alpha_{k -1}+2} \circ\ldots\circ f_{s,\alpha_{k}}
        \right),
    \]
    (for $1 \leq k \leq m$), where
    \[
        f_{s,\alpha_{k -1}+1}\circ f_{s,\alpha_{k -1}+2} \circ\ldots\circ f_{s,\alpha_{k}}
        \colon M_s^{\alpha_{k}}\to M_s^{\alpha_{k -1}}
    \]
    is between homeomorphic summands.

    Now, conditions (\ref{1}),(\ref{2}) and (\ref{3}) are all direct consequences of the construction; moreover,
    \[
        \sum_{k = 1}^m \beta_k =\sum_{k = 1}^m (\alpha_k -\alpha_{k -1})=\alpha_m -\alpha_0 \leq m \cdot r !\,.
    \]
\end{proof}

With the previous lemma we got a sequence of orientation preserving homotopy equivalences $(g_k \colon M \to M)_{1 \leq k \leq m}$ preserving the types of homeomorphism of some decomposing systems of spheres.
Thanks to this result, instead of considering the maps on the global manifold, we can study their action on the prime summands.
The following lemma is useful to understand how the maps $g_k$ act on the fundamental groups of the elliptic summands of the manifold $M$.

\begin{lemma}\label{GR}
    Let $G_1,\,\ldots ,\,G_r$ be finite groups of cardinality $c_1,\,\ldots ,\,c_r$ respectively, and set $m =  (c_1 \, \cdots \, c_r) ! \,$. Suppose that for any $s \in \{1,\,\ldots \,,\,r \}$ and $k \in \{1,\,\ldots \,,\,m \}$ an automorphism $\varphi_{s,k} \colon G_s \to G_s$ is given.  Then, there exist two integers $0 \leq n_1<n_2  \leq m$ such that the map
    \[\phi_{s,n_1 +1}\circ \phi_{s,n_{1}+{2}}\circ \,\ldots \, \circ \phi_{s,n_2}\colon G_s \to G_s \]
    is the identity $\id_{G_s}$ for every $s \in \{1,\,\ldots ,\,r \}$.
\end{lemma}
\begin{proof}
    We first prove the statement for $r = 1$: let $G$ be a finite group of cardinality $c$ and  $(\phi_k)_{1 \leq k \leq  c !}$ a (finite) sequence of automorphisms of $G$.
    Set
    \[\sigma_k = \phi_1 \circ \,\ldots \, \circ \phi_k \, \]
    for $k \in \{1,\,\ldots \,,\, c ! \}$ and $\sigma_0 =\id$.

    As $G$ has cardinality $c$, $\aut(G)$ has at most $c!$ elements.
    By the pigeonhole principle, there exist two integers $0\leq n_1 < n_2 \leq c !$ with the property that $\sigma_{n_1}= \sigma_{n_2}$. It follows
    \[ \phi_{n_1 +1}\circ \,...\,\circ\phi_{n_2} = \sigma_{n_1}^{-1}\circ\sigma_{n_2}= \id_G \,,\]
    as desired.
    
More generally, if $r \geq 1$, then set
    \[
        G = G_1 \times \,\ldots \,\times G_r,
    \]
    that is a finite group with $c= c_1 \,\cdots \,c_r$ elements, and consider for every $k \in \{1 ,\,\ldots \,,\, c ! \}$ the automorphism given by
    \[
        \phi_k = \left(\phi_{1,k},\,\ldots \,,\,\phi_{r,k}\right)\colon G \to G\,. \]
     The thesis is now a consequence of the case $r = 1$.
\end{proof}

We are ready to prove Theorem \ref{main}. As seen in Remark \ref{hypsim}, it is enough to prove it with the further assumptions that the manifold is connected and the homotopy equivalence is orientation preserving.

\begin{thm}\label{connect}
    Let $M = \csum M$ be a prime decomposition of an oriented closed connected $3$-manifold $M$ and suppose $M_1,\,\ldots \,,\,M_r$ have finite fundamental groups of cardinality $c_1,\,\ldots \,,\,c_r$ respectively, and $M_{r +1},\,\ldots \,,\, M_n$ have infinite fundamental groups.
    Set
    \[A_M = r !\cdot (c_1 \, \cdots \, c_r)!\,.
    \]
    Then for every orientation preserving homotopy equivalence $f \colon M \to M$ there exists an integer $k$ such that $1\leq k \leq A_M$ and $f^k$ is homotopic to a homeomorphism.
\end{thm}

\begin{proof}
    For every $s \in \{1,\,\ldots \,,\, r \}$ let us fix a base point $p_s \in M_s$ and set $G_s =\pi_1(M_s,p_s)$. Set $m = (c_1 \,\cdots \,c_r) !\,$.
    
    By applying Lemma \ref{lemmalungo} to the manifold $M$, the map $f$, and the integer $m$, we find $m +1$ decomposing systems of spheres $\Sigma_0,\,\ldots ,\,\Sigma_m$, $m$ positive integers $\beta_1,\,\ldots \,,\, \beta_m $ and $m$ orientation preserving self-homotopy equivalences $g_1,\,\ldots \,,\, g_m$ of $M$ such that for every $k \in\{1,\,\ldots \, ,\,m  \}$ the map $g_k$
    \begin{itemize}
        \item
            splits along $\Sigma_{k -1}$ and $g_k^{-1} ({\Sigma}_{k -1})= \Sigma_{k}$,
        \item
            preserves the types of homeomorphism of $\Sigma_{k -1}$,
        \item
            is homotopic to $f^{\beta_{k}}$.
    \end{itemize}
    Moreover, it holds that $\sum_{k = 1}^m \beta_k \leq m \cdot r !\,$.
    
    Let
    \[M = \csum{M^k}\]
    be the decomposition induced by $\Sigma_k$;
    without loss of generality, we can assume that for every $s \in \{1,\,\ldots \,,\,n \}$ and $k \in \{0,\,\ldots \,,\, m \} $ the summand $M_{s}^{k}$ is homeomorphic to $M_s $.

    We will find two integers $0 \leq n_1<n_2 \leq m$ such that the homotopy equivalence
    \[
        h = g_{n_1 +1}\,\circ \,g_{n_1 +2}\, \circ \,\ldots \,\circ \,g_{n_2}
    \]
    preserves the types of homeomorphism of $\Sigma_{n_1}$ and for every $s \in  \{1,\,\ldots \,,\,n\}$ the map

    \[\rrestr{h} {M_{s}^{n_2}}\colon M_{s}^{n_2}\to M_{s}^{n_1}\]
    is a homeomorphism.

    For every $s \in  \{1,\,\ldots \,,\,r\}$ there exists a finite sequence of orientation preserving homeomorphisms
    \[
        \left(\theta_{s,k} \colon M_{s}\to M_{s}^{k}  \right)_{0 \leq k \leq m}
    \]
    such that the map $\theta_{s,k -1}^{-1}\circ \rrestr{g_k} {M_{s}^{k}} \circ \theta_{s,k}$ is an orientation preserving self-homotopy equivalence of $M_s$ fixing the point $p_s$:
    indeed, for every $k \in \{0,\,\ldots \,,\,m \}$ we can choose $\theta_{s,k}'\colon M_{s}\to M_{s}^{k}$ to be any orientation preserving homeomorphism (it exists by construction), set $\theta_{s,m}=\theta_{s,m} '$ and then, by recursion on $k$ (starting with $k = m$ and finishing with $k = 0$),  we define $\theta_{s,k}$ to be the orientation preserving homeomorphism obtained from $\theta_{s,k}'$ by precomposing an orientation preserving self-homeomorphism of $M_s$ that sends the point $p_s$ to the point
    $(\theta'_{s,k })^{-1} \left(\rrestr{g_{k +1}} {M_{s}^{k +1}}\left(\theta_{s,k +1} (p_s)\right)\right)$.

    Let us denote by $\tau_{s,k}$ the orientation preserving homotopy equivalence given by
    \[
        \tau_{s,k}:=
        \left(\theta_{s,k -1}^{-1}\circ \rrestr{g_k} {M_{s}^{k}} \circ \theta_{s,k}\right)
        \colon M_s \to M_s \,.
    \]
    Since for every $s \in \{1,\,\ldots \,,\,r \}$ and $k \in \{0 ,\,\ldots \,,\, m \}$ the map $\tau_{s,k}$ fixes the point $p_s \in M_s$, it induces a well defined map
    \[\phi_{s,k}=(\tau_{s,k})_{\#}\colon G_s \to G_s \, \]
    on the fundamental group $\pi_1(M_s,p_s)= G_s$.
    As $\tau_{s,k}$ is a homotopy equivalence, the map $\phi_{s,k}$ is an automorphism of $G_s$.

    Now, Lemma  \ref{GR} applied to the finite groups $G_1,\,\ldots \,,\,G_r$ and the automorphisms $\phi_{s,k}\colon G_s \to G_s$ yields two integers $0 \leq n_1 < n_2 \leq m$
    such that  the map
    \[\phi_{s,n_1 +1}\circ \phi_{s,n_{1}+{2}}\circ \,\ldots \, \circ \phi_{s,n_2}\colon G_s \to G_s \]
    is the identity on $G_s$ for every $s \in \{1 ,\,\ldots \,,\, r \}$; since $M_s$ is elliptic, Theorem \ref{sma} implies the homotopy equivalence
    
    \[
        \tau_{s,n_1 +1}\circ \,\ldots \, \circ \tau_{s,n_2} =
        \theta_{s,n_1}^{-1}\circ \left(
            \rrestr{g_{n_1 +1}} {M_{s}^{n_1 +1}}\circ \ldots \circ \rrestr{g_{n_2}} {M_{s}^{n_2}}
        \right)\circ \,\theta_{s,n_2}
        \colon M_s \to M_s
    \]
    is homotopic to the identity of $M_s$.

    It follows that for every $s \in  \{1,\,\ldots \,,\,n\}$ the map
    \[
        \rrestr{g_{n_1 +1}} {M_{s}^{n_1 +1}}\circ \,\ldots \,\circ \rrestr{g_{n_2}} {M_{s}^{n_2}}\colon M_{s}^{n_2}\to M_{s}^{n_1}
    \]
    is homotopic to a homeomorphism $h_s$ (when $1 \leq s \leq r$, then we take $h_s = \theta_{s,n_1}\circ \theta_{s,n_2}^{-1}$ and when $r< s \leq n$, the existence of $h_s$ is ensured by the fact that $M_s$ is a strong Borel manifold).

    Let us set $t =\sum_{k = n_1 +1}^{n_2}\beta_k$;
    the map $f^{t} = f^{\beta_{n_1 +1}} \circ \,\ldots \,\circ f^{\beta_{n_2}}$ is homotopic to
    \[
        g_{n_1 +1} \circ \,\ldots \,\circ g_{n_2}=
        \mathop{\has}_{s=1}^n\left(
            \rrestr{g_{n_1 +1}} {M_{s}^{n_1 +1}}\circ \,\ldots \,\circ \rrestr{g_{n_2}} {M_{s}^{n_2}}
        \right)\,,
    \]
    which, in turn, is homotopic to the map $h=\mathop{\has}_{s=1}^n h_s$. Being a connected sum of homeomorphisms, the map $h$ is a homeomorphism as well. Moreover, by construction it holds that
    \[
    t = \sum_{k = n_1 +1}^{n_2}\beta_k \leq \sum_{k = 1}^m \beta_k \leq m \cdot r ! = r !\cdot (c_1 \,\cdots \,c_r)!=A_M \,.
   \]

\end{proof}

 \bibliography{biblio}
\bibliographystyle{alpha}

\end{document}